\documentclass{asl}
\usepackage{aliascnt,colonequals,mleftright}
\usepackage[hidelinks,unicode]{hyperref}

\newcommand{\newaliastheorem}[2]{%
  \newaliascnt{#1}{theorem}
  \newtheorem{#1}[#1]{#2}
  \aliascntresetthe{#1}
  \expandafter\def\csname #1autorefname\endcsname{#2}
}

\theoremstyle{plain}
\newtheorem{theorem}{Theorem}
\newaliastheorem{lemma}{Lemma}
\newaliastheorem{proposition}{Proposition}
\newaliastheorem{corollary}{Corollary}

\newtheorem*{fact}{Fact}

\newcommand{\tuple}[1]{\left\langle #1 \right\rangle}

\newcommand{\Ord}{\mathrm{Ord}}

\DeclareMathOperator{\dom}{dom}

\DeclareMathOperator{\TC}{TC}
\newcommand{\tc}[1]{\TC\mleft(\left\{ #1 \right\}\mright)}

\DeclareMathOperator{\dd}{def}

\title{Antichain of ordinals in intuitionistic set theory}

\def\subjclassname{\textup{2020} Mathematics Subject
Classification}
\subjclass{Primary: 03E70; Secondary: 03E10, 03F55}

\author{Shuwei Wang}
\revauthor{Wang, Shuwei}
\address{School of Mathematics\\
University of Leeds\\
Leeds LS2 9JT, UK}
\email{mmsw@leeds.ac.uk}

\begin{document}

\begin{abstract}
  In classical set theory, the ordinals form a linear chain that we often think of as a very thin portion of the set-theoretic universe. In intuitionistic set theory, however, this is not the case and there can be incomparable ordinals. In this paper, we shall show that starting from two incomparable ordinals, one can construct canonical bijections from any arbitrary set to an antichain of ordinals, and consequently any subset of the given set can be defined using ordinals as parameters. This implies the surprising result that in the theory ``$\mathrm{IKP} + {}$there exist two incomparable ordinals'', the statements $\Ord \subseteq L$ and $V = L$ are equivalent.
\end{abstract}

\maketitle

\section{Introduction}

This paper aims to report on a canonical way of constructing arbitrarily sized antichains of ordinals in intuitionistic set theory, developed during research on the behaviours of G\"odel's constructible universe $L$ in intuitionistic contexts. As surveyed by Lubarsky \cite{lubarsky93-intuitionistic-l} and more recently by Matthews and Rathjen \cite{matthews-rathjen24-constructible-universe}, the situation for $L$ is much more complicated than the classical case and, notably, it still remains open whether any intuitionistic set theory can prove the statement $\Ord \subseteq L$. In this paper, we present some results about intuitionistic ordinals that will suggest that this may not be as innocent a claim as it seems to be.

In classical set theory, the class $\Ord$ of \emph{ordinals}, i.e.\ transitive sets of transitive sets, satisfies the trichotomy
\[\forall \alpha, \beta \in \Ord \left(\alpha \in \beta \lor \alpha = \beta \lor \beta \in \alpha\right).\]
However, the same is not true in intuitionistic set theory. For example, if $\varphi$ is a formula such that the instance of excluded middle $\varphi \lor \neg \varphi$ does not hold, then none of the three disjuncts above can hold for the assignment $\alpha \colonequals 1$, $\beta \colonequals \left\{x \in 1 : \varphi\right\}$.

We will say that two ordinals $\alpha, \beta$ are \emph{incomparable} when the negation of this trichotomy holds, or more concisely written as
\[\alpha \perp \beta \quad \colonequals \quad \neg \alpha \in \beta^+ \land \neg \beta \in \alpha^+,\]
where $\alpha^+ = \alpha \cup \left\{\alpha\right\}$ denotes its ordinal \emph{successor}. We shall start by justifying formally that such ordinals may indeed exist in intuitionistic set theory\footnote{There are various other ways to justify the consistency of incomparable ordinals in different base theories such as $\mathrm{IKP}$ or $\mathrm{CZF}$. For example, the realisability model in Rathjen \cite{rathjen06-czf-realizability} contains plenty of incomparable ordinals. Here, we will present a construction that is most relevant to the final section of this paper.}:

\begin{proposition}
  \label{prop:perp-consistent}
  Assuming that $\mathrm{IZF}$ is consistent, then so is the theory
  \[\mathrm{IZF} + \exists \alpha, \beta \in \Ord \ \alpha \perp \beta.\]
\end{proposition}

\begin{proof}
  Suppose otherwise, that is, $\mathrm{IZF} \vdash \neg \exists \alpha, \beta \in \Ord \ \alpha \perp \beta$. Specifically, for an arbitrary formula $\psi\mleft(x_1, \ldots, x_n\mright)$ with all free variables listed, we can consider the set
  \[\alpha_\psi = 2 \cup \left\{\beta \in \mathcal{P}\mleft(1\mright) : \exists x_1, \ldots, x_n \left(0 \in \beta \leftrightarrow \psi\mleft(x_1, \ldots, x_n\mright)\right)\right\}.\]
  It is easy to see that
  \[\mathrm{IZF} \vdash \alpha_\psi \in \Ord \land \neg \alpha_\psi \in 2 \land \neg 2 \in \alpha_\psi.\]
  Additionally, for any $x_1, \ldots, x_n$ there exists a set
  \[\beta = \left\{\gamma \in 1 : \psi\mleft(x_1, \ldots, x_n\mright)\right\} \in \alpha_\psi\]
  by comprehension in $\mathrm{IZF}$, i.e.\ $0 \in \beta \leftrightarrow \psi\mleft(x_1, \ldots, x_n\mright)$ and thus $\beta \in 2$ if and only if $\beta = 0 \lor \beta = 1$, if and only if
  \[\neg \psi\mleft(x_1, \ldots, x_n\mright) \lor \psi\mleft(x_1, \ldots, x_n\mright).\]
  In other words,
  \[\mathrm{IZF} \vdash \alpha_\psi = 2 \leftrightarrow \forall x_1, \ldots, x_n \left(\psi\mleft(x_1, \ldots, x_n\mright) \lor \neg \psi\mleft(x_1, \ldots, x_n\mright)\right).\]
  By our assumption, we now have $\mathrm{IZF} \vdash \neg \alpha_\psi \perp 2$. Then we must have $\mathrm{IZF} \vdash \neg \neg \alpha_\psi = 2$ and thus the corresponding double negated instance of excluded middle:
  \[\mathrm{IZF} \vdash \neg \neg \forall x_1, \ldots, x_n \left(\psi\mleft(x_1, \ldots, x_n\mright) \lor \neg \psi\mleft(x_1, \ldots, x_n\mright)\right).\]

  However, this is impossible. We say that $\varphi$ is an \emph{anti-classical axiom} if $\mathrm{ZF} \vdash \neg \varphi$, then it is well-known that $\mathrm{IZF}$ is consistent with many anti-classical axioms. For example, see the realisability model in \cite{chen-rm-rathjen12-lifschitz-realizability} interpreting such axioms like Church's thesis, etc. Now, let $\varphi$ be any such anti-classical axiom and $\chi_1, \ldots, \chi_k$ be an enumeration of the necessary instances of excluded middle such that
  \[\mathrm{IZF} + \chi_1 + \cdots + \chi_k \vdash \neg \varphi,\]
  then by the analysis above we must have $\mathrm{IZF} \vdash \neg \neg \chi_i$ for all $i$ and consequently $\mathrm{IZF} \vdash \neg \neg \neg \varphi$. This is a contradiction to the consistency of $\varphi$ and thus our initial assumption must be false.
\end{proof}

In the main portion of this paper, we shall proceed to show that such incomparable ordinals are highly pathological objects --- this is probably hinted at in Lubarsky's recent survey \cite{lubarsky23-inner-outer-models-cst}, section 6, where he claimed ``there are [unpublished] constructions showing that it is not hard to get an arbitrary set into $L$ by coding it into an ordinal'' --- and we shall make it precise here that, from merely a pair of two incomparable ordinals, one can construct an injection $f$ from an arbitrarily large domain to an \emph{antichain} of ordinals, that is, for any $x \neq y \in \dom\mleft(f\mright)$, we will have $f\mleft(x\mright) \perp f\mleft(y\mright)$.

Additionally, we construct this map $f$ as a canonical $\Sigma$-definable object in the language of set theory. It induces a method to uniformly encode any subsets of its domain as ordinals and is now even preserved in the constructible universe. We shall utilise this encoding to prove the following result, which demonstrates how intuitionistic $L$ behaves vastly different from its classical counterpart:

\begin{theorem}[$\mathrm{IKP}$]
  \label{thm:main-equiv}
  Assume there exist $\alpha \perp \beta \in \Ord$, then the following are equivalent:
  \begin{itemize}
    \item $\Ord \subseteq L$,
    \item $V = L$.
  \end{itemize}
\end{theorem}

One immediate consequence of this, as we shall comment at the end, is that the consistency of any theory $\mathrm{IZF} + \varphi + V \neq L$, where $\varphi$ is any anti-classical axiom, will imply $\mathrm{IZF} \nvdash \Ord \subseteq L$, and will resolve the long-standing open problem mentioned above.

\section{Constructing antichains}

A very weak intuitionistic set theory, $\mathrm{IKP}$, suffices for our construction, which means the same results easily apply to any intuitionistic extensions such as $\mathrm{CZF}$ or $\mathrm{IZF}$.  Essentially, the construction uses the fact that $\mathrm{IKP}$ proves $\Sigma$-recursion and also separation by any $\Delta_0$-formula containing $\Sigma$-definable class functions. We also assume that $\mathrm{IKP}$ here includes the axiom of strong infinity, which is needed for the existence of transitive closures that will become convenient later.

Like in a classical set theory, we formally say that a set $\alpha$ is an \emph{ordinal} if it is a transitive set of transitive sets; the class of all ordinals is denoted $\Ord$. We cannot intuitionistically classify all ordinals into the cases of zero, successors and limits. Thus, to obtain the addition operator $\alpha + \cdot : \Ord \rightarrow \Ord$, we need to modify the recursive definition to use a uniform clause. We let
\[\alpha + \beta = \alpha \cup \bigcup_{\gamma \in \beta} \left(\alpha + \gamma\right)^+,\]
which is common practice in intuitionistic contexts such as in \cite{aczel-rathjen10-cst-book}.

To begin, we shall verify that a few basic classical properties of ordinal addition still hold in intuitionistic set theory. $\mathrm{IKP}$ uses the weaker axiom of set induction instead of foundation, but it still forbids self-containing sets:

\begin{lemma}
  \label{lem:no-in-self}
  There is no set $x$ such that $x \in x$.
\end{lemma}

\begin{proof}
  Suppose that such an element $x$ exists. Observe that for any set $y$, $y = x$ implies $x \in y$, i.e.\ $\exists z \in y \ z = x$. Thus, $\forall z \in y \ \neg z = x$ implies $\neg y = x$. By set induction, we must have $\forall y \ \neg y = x$, which is a contradiction.
\end{proof}

\begin{lemma}
  \label{lem:ord-plus-not-in-original}
  For any $\alpha, \beta \in \Ord$, we have $\neg \alpha + \beta \in \alpha$.
\end{lemma}

\begin{proof}
  Suppose that $\alpha + \beta \in \alpha$, then $\alpha + \beta \in \alpha \cup \bigcup_{\gamma \in \beta} \left(\alpha + \gamma\right)^+ = \alpha + \beta$, contradicting \autoref{lem:no-in-self}.
\end{proof}

\begin{lemma}
  \label{lem:ord-plus-inj}
  For any $\alpha, \beta, \gamma \in \Ord$, we have $\alpha + \beta \in \alpha + \gamma \rightarrow \beta \in \gamma$ and $\alpha + \beta = \alpha + \gamma \rightarrow \beta = \gamma$.
\end{lemma}

\begin{proof}
  We prove the conjunction of the two desired properties by simultaneous set induction on the parameters $\beta$ and $\gamma$. First suppose that
  \[\alpha + \beta \in \alpha + \gamma = \alpha \cup \bigcup_{\delta \in \gamma} \left(\alpha + \delta\right)^+.\]
  By \autoref{lem:ord-plus-not-in-original}, $\neg \alpha + \beta \in \alpha$, thus we must have $\alpha + \beta \in \left(\alpha + \delta\right)^+$ for some $\delta \in \gamma$, i.e.\ $\alpha + \beta \in \alpha + \delta \lor \alpha + \beta = \alpha + \delta$. We thus must have $\beta \in \delta \lor \beta = \delta$ by the inductive hypothesis and thus $\beta \in \gamma$ either way.

  On the other hand, suppose that $\alpha + \beta = \alpha + \gamma$, then
  \[\alpha + \delta \in \left(\alpha + \delta\right)^+ \subseteq \alpha + \beta = \alpha + \gamma\]
  for any $\delta \in \beta$. By the inductive hypothesis, we have $\delta \in \gamma$ for any $\delta \in \beta$, i.e.\ $\beta \subseteq \gamma$. An entirely symmetrical argument shows that $\gamma \subseteq \beta$ as well, hence $\beta = \gamma$ as desired.
\end{proof}

Recall from the introduction that incomparability is defined as
\[\alpha \perp \beta \quad \colonequals \quad \neg \alpha \in \beta^+ \land \neg \beta \in \alpha^+.\]
We check that this is preserved by carefully constructed sums:

\begin{lemma}
  \label{lem:perp-plus-still-perp}
  When $\alpha \perp \beta \in \Ord$, for any $\gamma \in \Ord$, also $\alpha \perp \beta^+ + \gamma$.
\end{lemma}

\begin{proof}
  Firstly, suppose that $\beta^+ + \gamma \in \alpha^+$. Then $\beta \in \beta^+ + \gamma \subseteq \alpha$, contradicting our assumption that $\alpha \perp \beta$. Observe that this already implies $\neg \alpha = \beta^+ + \gamma$ for any $\gamma \in \Ord$.

  It suffices to further show by set induction on $\gamma$ that $\neg \alpha \in \beta^+ + \gamma$. Suppose that $\alpha \in \beta^+ + \gamma$ yet $\neg \alpha \in \beta^+ + \delta$ for every $\delta \in \gamma$. Since also $\neg \alpha = \beta^+ + \delta$, i.e.\ we have $\neg \alpha \in \left(\beta^+ + \delta\right)^+$, thus we must have $\alpha \in \beta^+$. This again contradicts $\alpha \perp \beta$.
\end{proof}

\begin{corollary}
  \label{cor:ord-pair}
  When $\alpha \perp \beta \in \Ord$, for any $\gamma_1, \gamma_2, \delta_1, \delta_2 \in \Ord$, we have
  \[\left(\alpha^+ + \gamma_1\right) \cup \left(\beta^+ + \gamma_2\right) = \left(\alpha^+ + \delta_1\right) \cup \left(\beta^+ + \delta_2\right) \rightarrow \gamma_1 = \delta_1 \land \gamma_2 = \delta_2.\]
\end{corollary}

\begin{proof}
  Assume $\left(\alpha^+ + \gamma_1\right) \cup \left(\beta^+ + \gamma_2\right) = \left(\alpha^+ + \delta_1\right) \cup \left(\beta^+ + \delta\right)$. For any $\eta \in \gamma_1$, we have $\alpha^+ + \eta \in \alpha^+ + \gamma_1 \subseteq \left(\alpha^+ + \delta_1\right) \cup \left(\beta^+ + \delta_2\right)$. Suppose that $\alpha^+ + \eta \in \beta^+ + \delta_2$, then $\alpha \in \beta^+ + \delta_2$ by transitivity, and this contradicts \autoref{lem:perp-plus-still-perp}. This means we must have $\alpha^+ + \eta \in \alpha^+ + \delta_1$ and thus $\eta \in \delta_1$ by \autoref{lem:ord-plus-inj}, i.e.\ $\gamma_1 \subseteq \delta_1$. Entirely symmetrical arguments show that $\delta_1 \subseteq \gamma_1$, $\gamma_2 \subseteq \delta_2$ and $\delta_2 \subseteq \gamma_2$ as well, hence $\gamma_1 = \delta_1$ and $\gamma_2 = \delta_2$ as desired.
\end{proof}

Observe that \autoref{cor:ord-pair} essentially claims that given two incomparable ordinals, there is a method to represent any pair of two ordinals as a single ordinal uniformly.

Now, before proceeding to treat antichains, we observe that the classical way to define an antichain, i.e.
\[\forall a, b \in \dom\mleft(f\mright) \left(a \neq b \leftrightarrow f\mleft(a\mright) \perp f\mleft(b\mright)\right),\]
is intuitionistically ``weak'' because the both sides of the biconditional consist of negated formulae, so we cannot derive positive statements like $a = b$ when we need them since intuitionistic logic lacks double-negation elimination. Thus, we choose to adopt an alternative definition and say that for any set $x$, a function $f : x \rightarrow \Ord$ on it is \emph{pairwise incomparable} if
\[\forall a, b \in x \left(\neg f\mleft(a\mright) \in f\mleft(b\mright) \land \left(f\mleft(a\mright) = f\mleft(b\mright) \rightarrow a = b\right)\right).\]
This is classical equivalent to the biconditional above, but intuitionistically stronger.

\begin{proposition}
  \label{prop:perp-ineq-to-perp}
  When $\alpha \perp \beta \in \Ord$, for any set $s \subseteq \Ord$ of ordinals, the function $f : s \rightarrow \Ord$ given by
  \[f\mleft(\gamma\mright) = \alpha^+ \cup \left(\beta^+ + \gamma\right)\]
  is pairwise incomparable.
\end{proposition}

\begin{proof}
  We first show that $\neg f\mleft(\gamma\mright) \in f\mleft(\delta\mright)$ for any $\gamma, \delta \in s$. Suppose that $\alpha^+ \cup \left(\beta^+ + \gamma\right) \in \alpha^+ \cup \left(\beta^+ + \delta\right)$. Notice that $\alpha^+ \cup \left(\beta^+ + \gamma\right) \in \alpha^+$ implies $\alpha \in \alpha^+ \cup \left(\beta^+ + \gamma\right) \subseteq \alpha$, contradicting \autoref{lem:no-in-self}, thus we must have $\alpha^+ \cup \left(\beta^+ + \gamma\right) \in \beta^+ + \delta$ instead. However, by \autoref{lem:perp-plus-still-perp}, we have $\neg \alpha \in \beta^+ + \delta$ and hence another contradiction. Thus we must have $\neg \alpha^+ \cup \left(\beta^+ + \gamma\right) \in \alpha^+ \cup \left(\beta^+ + \delta\right)$.

  Now, it remains to show that
  \[\alpha^+ \cup \left(\beta^+ + \gamma\right) = \alpha^+ \cup \left(\beta^+ + \delta\right) \rightarrow \gamma = \delta.\]
  This is simply an instance of \autoref{cor:ord-pair} where $\gamma_1 = \delta_1 = \varnothing$.
\end{proof}

Given fixed $\alpha \perp \beta \in \Ord$ and any function $f : x \rightarrow \Ord$, we shall define $f^\perp : x \rightarrow \Ord$ by
\[f^\perp\mleft(y\mright) = \alpha ^+ \cup \left(\beta^+ + f\mleft(y\mright)\right).\]
Observe that by \autoref{prop:perp-ineq-to-perp}, $f^\perp$ is pairwise incomparable if and only if $f$ is injective, i.e.\ $f$ satisfies $\forall y, z \in x \left(f\mleft(y\mright) = f\mleft(z\mright) \rightarrow y = z\right)$.

\begin{proposition}
  \label{prop:pairwise-perp-powerset-ord}
  Let $x$ be any set and $f : x \rightarrow \Ord$ be pairwise incomparable. Then for any $y \subseteq x$,
  \[\forall z \in x \left(z \in y \leftrightarrow f\mleft(z\mright) \in \bigcup_{t \in y} f\mleft(t\mright)^+\right).\]
\end{proposition}

\begin{proof}
  The forward direction is trivial. For the backward direction, suppose that $f\mleft(z\mright) \in \bigcup_{t \in y} f\mleft(t\mright)^+$, i.e.\ $f\mleft(z\mright) \in f\mleft(t\mright)^+$ for some $t \in y$. From the pairwise incomparability condition, we know that $\neg f\mleft(z\mright) \in f\mleft(t\mright)$, so we must have $f\mleft(z\mright) = f\mleft(t\mright)$, which immediately implies $z = t \in y$.
\end{proof}

\begin{corollary}
  \label{cor:perp-to-pow}
  Fix $\alpha \perp \beta \in \Ord$. Let $x$ be any set and $f : x \rightarrow \Ord$ be pairwise incomparable. Then for any set\footnote{We work in $\mathrm{IKP}$ which does not assert the axiom of powerset, thus the notation $\mathcal{P}\mleft(x\mright)$ will be treated as a $\Delta_0$-class (which may not be a set) in our context.} $y \subseteq \mathcal{P}\mleft(x\mright)$, there exists $g : y \rightarrow \Ord$ that is also pairwise incomparable.
\end{corollary}

\begin{proof}
  We define $h : y \rightarrow \Ord$ as $z \mapsto \bigcup_{t \in z} f\mleft(t\mright)^+$. Then by \autoref{prop:pairwise-perp-powerset-ord},
  \[\forall r, s \in y \left(h\mleft(r\mright) = h\mleft(s\mright) \rightarrow \forall t \left(t \in r \leftrightarrow t \in s\right)\right),\]
  i.e.\ $h$ is injective. Therefore, $g = h^\perp$ is pairwise incomparable.
\end{proof}

Additionally, using \autoref{cor:ord-pair} we can merge ordinal-indexed families of pairwise incomparable functions:

\begin{proposition}
  \label{prop:perp-to-ord-seq-union}
  Fix $\alpha \perp \beta \in \Ord$. Consider some set $s \subseteq \Ord$ and a function $a$ with domain $s$. For any function $f$ with domain $s$ such that $f\mleft(\gamma\mright) : a\mleft(\gamma\mright) \rightarrow \Ord$ is a pairwise incomparable function for each $\gamma \in s$, then there exists $g : \bigcup_{\gamma \in s} a\mleft(\gamma\mright) \rightarrow \Ord$ that is also pairwise incomparable.
\end{proposition}

\begin{proof}
  Consider the set of pairs
  \[b = \left\{\tuple{\gamma, x} : \gamma \in s \land x \in a\mleft(\gamma\mright)\right\}.\]
  We define $\tilde{f} : b \rightarrow \Ord$ as $\tuple{\gamma, x} \mapsto \left(\alpha^+ + \gamma\right) \cup \left(\beta^+ + f\mleft(\gamma\mright)\mleft(x\mright)\right)$. By \autoref{cor:ord-pair} and the fact that each $f\mleft(\gamma\mright)$ is pairwise incomparable,
  \[\tilde{f}\mleft(\tuple{\gamma, x}\mright) = \tilde{f}\mleft(\tuple{\delta, y}\mright) \rightarrow \gamma = \delta \land x = y\]
  for any $\tuple{\gamma, x}, \tuple{\delta, y} \in b$, so $\tilde{f}^\perp$ is pairwise incomparable. Now, the set
  \[c = \left\{b_x : x \in \bigcup_{\gamma \in s} a\mleft(\gamma\mright)\right\},\]
  where $b_x = \left\{t \in b : \exists \gamma \in s \ t = \tuple{\gamma, x}\right\}$, exists by $\Delta_0$-replacement. Also $c \subseteq \mathcal{P}\mleft(b\mright)$, so by \autoref{cor:perp-to-pow}, there is $h : c \rightarrow \Ord$ that is pairwise incomparable.

  We now construct $g : \bigcup_{\gamma \in s} a\mleft(\gamma\mright) \rightarrow \Ord$ as $x \mapsto h\mleft(b_x\mright)$. It suffices to verify $\forall x, y \in \bigcup_{\gamma \in s} a\mleft(\gamma\mright) \left(g\mleft(x\mright) = g\mleft(y\mright) \rightarrow x = y\right)$. We know that $g\mleft(x\mright) = g\mleft(y\mright)$ implies $b_x = b_y$, but we also know that there exists $\gamma \in s$ such that $x \in a\mleft(\gamma\mright)$. This means $\tuple{\gamma, x} \in b_x = b_y$, i.e.\ there exists $\delta \in s$ such that $\tuple{\gamma, x} = \tuple{\delta, y}$. It follows immediately that $x = y$.
\end{proof}

Finally, fix any set $x$ and its transitive closure is given as
\[\tc{x} = \left\{x\right\} \cup x \cup \bigcup x \cup \bigcup \bigcup x \cup \cdots\]
by recursion on $\omega$. Assuming that there exists incomparable ordinals, we will construct a pairwise incomparable function with domain $\tc{x}$ canonically.

To this end, we shall divide the transitive closure into a hierarchy
\[\tc{x}_\alpha = \left\{y \in \tc{x} : \exists \beta \in \alpha \ y \subseteq \tc{x}_\beta\right\}.\]
The following lemma immediately shows that $\tc{x} = \bigcup_{\alpha \in \Ord} \tc{x}_\alpha$:

\begin{lemma}
  \label{lem:tc-hier-contains-all}
  For any set $x$ and any $y \in \tc{x}$, there exists $\alpha \in \Ord$ such that $y \subseteq \tc{x}_\alpha$, i.e.\ $y \in \tc{x}_{\alpha^+}$.
\end{lemma}

\begin{proof}
  We prove this by set induction on $y$. Assume $y \in \tc{x}$ and suppose that
  \[\forall z \in y \ \left(z \in \tc{x} \rightarrow \exists \alpha \in \Ord \ z \subseteq \tc{x}_\alpha\right).\]
  By transitivity $y \subseteq \tc{x}$, thus by strong $\Sigma$-collection we have a set $s \subseteq \Ord$ such that $\forall z \in y \ \exists \alpha \in s \ z \subseteq \tc{x}_\alpha$. Take $\beta = \bigcup_{\alpha \in s} \alpha^+$, then $s \subseteq \beta$ and hence $y \subseteq \tc{x}_\beta$.
\end{proof}

It is easy to verify that $\forall \alpha, \beta \in \Ord \left(\alpha \subseteq \beta \rightarrow \tc{x}_\alpha \subseteq \tc{x}_\beta\right)$, thus
\[\tc{x}_{\alpha^+} = \tc{x} \cap \mathcal{P}\mleft(\tc{x}_\alpha\mright),\]
and also $\forall \alpha \in \Ord \ \tc{x}_\alpha = \bigcup_{\beta \in \alpha} \tc{x}_{\beta^+}$.

Observe that in the proofs of both \autoref{cor:perp-to-pow} and \autoref{prop:perp-to-ord-seq-union}, we gave explicit constructions of the desired $g$ from $f$, and it is not hard to verify that the described functions can be uniquely defined from $f$ using a $\Sigma$-formula. Fix $\alpha \perp \beta \in \Ord$ and we can iterate the following process: given a family $\left(f_\delta\right)_{\delta \in \gamma}$ where each $f_\delta : \tc{x}_\delta \rightarrow \Ord$ is a pairwise incomparable function, then we collect the family $\left(g_\delta\right)_{\delta \in \gamma}$ of functions $g_\delta : \tc{x}_{\delta^+} \rightarrow \Ord$ built in \autoref{cor:perp-to-pow} through $\Sigma$-replacement and thus get a pairwise incomparable function $f_\gamma : \tc{x}_\gamma \rightarrow \Ord$ by \autoref{prop:perp-to-ord-seq-union}. By $\Sigma$-recursion, we then have a $\Sigma$-definable class-family $\left(f_{x, \gamma}\right)_{\gamma \in \Ord}$ where each $f_{x, \gamma} : \tc{x}_\gamma \rightarrow \Ord$ is pairwise incomparable.

Additionally, by \autoref{lem:tc-hier-contains-all}, there must exist some $\gamma \in \Ord$ such that $x \in \tc{x}_\gamma$. Since it is easy to verify that each $\tc{x}_\alpha$ is transitive, we must have
\[x \subseteq \tc{x} = \tc{x}_\gamma.\]
Consequently, $f_{x, \gamma}$ is a pairwise incomparable function with domain $\tc{x}$, and its restriction $\left.f_{x, \gamma}\right|_x$ will immediately provide a bijection between $x$ and an antichain of ordinals, as claimed in the introduction. This function will be $\Sigma$-definable using the incomparable ordinals $\alpha \perp \beta \in \Ord$, the domain $x$ and an arbitrary (large enough) ordinal $\gamma$ as parameters.

\section{The case inside \texorpdfstring{$L$}{L}}

G\"odel's constructible universe $L$ can be defined\footnote{We always assume that $\mathrm{IKP}$ contains the axiom of infinity, so we do not need to worry about working with countable sets of formulae or recursion on $\omega$. For a more intricate treatment of $L$ that is still viable without the axiom of infinity, refer to section 4--5 of \cite{matthews-rathjen24-constructible-universe}.} in $\mathrm{IKP}$ in the usual way, by formalising the construction of $\dd\mleft(A\mright)$, the collection of first-order definable subsets (with parameters) of $\tuple{A; \in}$. We can then let
\[L_\alpha = \bigcup_{\beta \in \alpha} \dd\mleft(L_\beta\mright)\]
for all $\alpha \in \Ord$ recursively and finally $L = \bigcup_{\alpha \in \Ord} L_\alpha$.

As mentioned, basic properties of $L$ in an intuitionistic set theory were treated by Lubarsky \cite{lubarsky93-intuitionistic-l} and more recently by Matthews and Rathjen \cite{matthews-rathjen24-constructible-universe}. Notably, in order to prove the main result of this paper, we will use Theorem 5.7 in \cite{matthews-rathjen24-constructible-universe} that

\begin{fact}
  For every axiom $\varphi$ of $\mathrm{IKP}$, $\mathrm{IKP} \vdash L \vDash \varphi$.
\end{fact}

This means, starting with two incomparable ordinals $\alpha \perp \beta$ that are already in $L$, the construction of the functions $f_{x, \gamma}$ above can be replicated in $L$. Since the functions are defined by a $\Sigma$-formula, which is upwards absolute, we immediately know that
\[\forall x \in L \ \forall \gamma \in L \cap \Ord \ f_{x, \gamma}^L = f_{x, \gamma}.\]

We need to verify an additional lemma, that each specific value $f_{x, \gamma}\mleft(y\mright)$ of the function does not actually depend on $x$:

\begin{lemma}
  \label{lem:tc-hier-and-f-agree}
  Fix $\alpha \perp \beta \in \Ord$. For any set $x$, any $y \in \tc{x}$ and any $\gamma \in \Ord$, $y \in \tc{x}_\gamma$ if and only if $y \in \tc{y}_\gamma$, and we also have $f_{x, \gamma}\mleft(y\mright) = f_{y, \gamma}\mleft(y\mright)$ when this holds.
\end{lemma}

\begin{proof}
  Observe that $\tc{y}_\gamma \subseteq \tc{x}_\gamma$ follows easily from $\tc{y} \subseteq \tc{x}$ by induction on $\gamma$. For the other direction, we prove by set induction on $y$. Assume $y \in \tc{x}_\gamma$, then there exists $\delta \in \gamma$ such that $y \subseteq \tc{x}_\delta$, but then the inductive hypothesis implies that
  \[y \subseteq \bigcup_{z \in y} \tc{z}_\delta \subseteq \tc{y}_\delta\]
  as well.

  Additionally, it is easy to check with the explicit constructions in \autoref{cor:perp-to-pow} and \autoref{prop:perp-to-ord-seq-union} that the values of $f_{x, \gamma}\mleft(y\mright)$ and $f_{y, \gamma}\mleft(y\mright)$ only depend on the values of $f_{x, \delta}\mleft(z\mright)$ and $f_{y, \delta}\mleft(z\mright)$ respectively for any $\delta \in \gamma$ and $z \in y$ such that $z \in \tc{x}_\delta$. We again use an induction on $y$, and notice that the inductive hypothesis implies that for such $\delta$ and $z$ we always have
  \[f_{x, \delta}\mleft(z\mright) = f_{z, \delta}\mleft(z\mright) = f_{y, \delta}\mleft(z\mright),\]
  thus $f_{x, \gamma}\mleft(y\mright) = f_{y, \gamma}\mleft(y\mright)$ as well.
\end{proof}

In other words, the value $f_{x, \gamma}\mleft(y\mright)$ is determined by $y$ and the ordinal parameters $\alpha, \beta, \gamma$ alone. We can finally prove the main result from the introduction:

\edef\thmtmp{\thetheorem}
\setcounter{theorem}{1}
\begin{theorem}
  Assume there exist $\alpha \perp \beta \in \Ord$, then the following are equivalent:
  \begin{itemize}
    \item $\Ord \subseteq L$,
    \item $V = L$.
  \end{itemize}
\end{theorem}
\setcounter{theorem}{\thmtmp}

\begin{proof}
  The backward direction is trivial. For the forward direction, we use set induction on $x$ to show that $\forall x \ x \in L$.

  From the inductive hypothesis, we will know that $x \subseteq L$, so by strong $\Sigma$-collection we have a set $s \subseteq \Ord$ such that $\forall y \in x \ \exists \gamma \in s \ y \in L_\gamma$. Take $\sigma = \bigcup_{\gamma \in s} \gamma$, then $x \subseteq L_\sigma$.

  Now, by \autoref{lem:tc-hier-contains-all}, there is some $\gamma \in \Ord$ such that $x \in \tc{x}_{\gamma^+}$ (and we know $\gamma \in L$ from our assumption). For any $y \in x$, we then know from \autoref{lem:tc-hier-and-f-agree} that $y \in \tc{y}_\gamma$ and
  \[f_{x, \gamma}\mleft(y\mright) = f_{y, \gamma}\mleft(y\mright) = f_{y, \gamma}^L\mleft(y\mright).\]
  We can now consider the set
  \[z = \left\{t \in L_\sigma : t \in \tc{t}_\gamma \land f_{t, \gamma}^L\mleft(t\mright) \in \tau\right\}\]
  where $\tau = \bigcup_{y \in x} f_{y, \gamma}\mleft(y\mright)^+$ is an ordinal. Assuming that all ordinals are in $L$, then by $\Delta_0$-separation, $z \in L$.

  It is immediate that $x \subseteq z$. For any $t \in z$, there must exist some $y \in x$ such that $f_{t, \gamma}\mleft(t\mright) \in f_{y, \gamma}\mleft(y\mright)^+$. Consider the set $p = \left\{t, y\right\}$ then we have $t, y \in \tc{p}_\gamma$ by \autoref{lem:tc-hier-and-f-agree} and
  \[f_{p, \gamma}\mleft(t\mright) = f_{t, \gamma}\mleft(t\mright) \in f_{y, \gamma}\mleft(y\mright)^+ = \bigcup_{s \in \left\{y\right\}} f_{p, \gamma}\mleft(s\mright)^+.\]
  Here $f_{p, \gamma} : \tc{p}_\gamma \rightarrow \Ord$ is pairwise incomparable and by \autoref{prop:pairwise-perp-powerset-ord}, this implies $t \in \left\{y\right\}$, i.e.\ $t = y \in x$. It follows that $x = z \in L$ as desired.
\end{proof}

\section{A final remark}

Using the same constructions as in \autoref{prop:perp-consistent}, we observe that \autoref{thm:main-equiv} implies finding any non-classical model of $\mathrm{IZF} + V \neq L$ would immediate disprove $\mathrm{IZF} \vdash \Ord \subseteq L$. More precisely,

\begin{corollary}
  Let $\varphi$ be any anti-classical axiom, that is, $\mathrm{ZF} \vdash \neg \varphi$. If $\mathrm{IZF} + \varphi + V \neq L$ is consistent, then $\mathrm{IZF} \nvdash \Ord \subseteq L$.
\end{corollary}

\begin{proof}
  We assume that $\mathrm{IZF} \vdash \Ord \subseteq L$. Then by \autoref{thm:main-equiv},
  \[\mathrm{IZF} + V \neq L \vdash \neg \exists \alpha, \beta \in \Ord \ \alpha \perp \beta.\]
  Looking at the same ordinals $\alpha_\psi$ and $2$ as in the proof of \autoref{prop:perp-consistent}, we would then conclude that $\mathrm{IZF} + V \neq L$ is inconsistent with any anti-classical axioms. This contradicts the consistency of $\mathrm{IZF} + \varphi + V \neq L$.
\end{proof}

However, constructing intuitionistic models of $V \neq L$ has proven to be very difficult, precisely due to the complicated structure of ordinals and G\"odel's $L$ discussed in this paper, and no such consistency results have been known to the author. We hope that further research will be able to settle the problem regarding $\Ord \subseteq L$ in intuitionistic set theories in this direction.

\section*{Funding}\normalfont

During the preparation of this paper the author is supported by the UK Engineering and Physical Sciences Research Council [EP/W523860/1].

\bibliographystyle{asl}
\bibliography{\jobname}

\end{document}